\documentclass[a4paper]{amsart}

\usepackage{amssymb}
\usepackage{amscd}
\usepackage{enumitem}
\usepackage{hyperref}
\usepackage[utf8]{inputenc}
\usepackage{newunicodechar}
\usepackage[arrow,curve,matrix]{xy}

\usepackage{colortbl}
\usepackage{graphicx}
\usepackage{tikz}

\usepackage{mathrsfs}

\definecolor{linkred}{rgb}{0.7,0.2,0.2}
\definecolor{linkblue}{rgb}{0,0.2,0.6}

\numberwithin{figure}{section}

\newunicodechar{α}{\ensuremath{\alpha}}
\newunicodechar{β}{\ensuremath{\beta}}
\newunicodechar{χ}{\ensuremath{\chi}}
\newunicodechar{δ}{\ensuremath{\delta}}
\newunicodechar{ε}{\ensuremath{\varepsilon}}
\newunicodechar{Δ}{\ensuremath{\Delta}}
\newunicodechar{η}{\ensuremath{\eta}}
\newunicodechar{γ}{\ensuremath{\gamma}}
\newunicodechar{Γ}{\ensuremath{\Gamma}}
\newunicodechar{ι}{\ensuremath{\iota}}
\newunicodechar{κ}{\ensuremath{\kappa}}
\newunicodechar{λ}{\ensuremath{\lambda}}
\newunicodechar{Λ}{\ensuremath{\Lambda}}
\newunicodechar{μ}{\ensuremath{\mu}}
\newunicodechar{ν}{\ensuremath{\nu}}
\newunicodechar{ω}{\ensuremath{\omega}}
\newunicodechar{Ω}{\ensuremath{\Omega}}
\newunicodechar{π}{\ensuremath{\pi}}
\newunicodechar{φ}{\ensuremath{\phi}}
\newunicodechar{Φ}{\ensuremath{\Phi}}
\newunicodechar{ψ}{\ensuremath{\psi}}
\newunicodechar{Ψ}{\ensuremath{\Psi}}
\newunicodechar{ρ}{\ensuremath{\rho}}
\newunicodechar{σ}{\ensuremath{\sigma}}
\newunicodechar{Σ}{\ensuremath{\Sigma}}
\newunicodechar{τ}{\ensuremath{\tau}}
\newunicodechar{θ}{\ensuremath{\theta}}
\newunicodechar{Θ}{\ensuremath{\Theta}}
\newunicodechar{ξ}{\ensuremath{\xi}}
\newunicodechar{ζ}{\ensuremath{\zeta}}

\newcommand{\OO}{\mathcal{O}}

\newcommand{\sR}{\mathcal{R}}

\newcommand{\sU}{\mathcal{U}}

\newcommand{\Zar}{\mathrm{Zar}}

\newcommand{\Frac}{\mathrm{Frac}} % Quotient field
\newcommand{\Aff}{\mathrm{Aff}}

\newcommand{\Sch}{\mathrm{Sch}}

\newcommand{\perf}{\mathrm{perf}}

\newcommand{\Shv}{\mathrm{Shv}}

\newcommand{\Set}{\mathrm{Set}}

\newcommand{\Spec}{\mathrm{Spec}}

\newcommand{\red}{\mathrm{red}}

\theoremstyle{theorem}
\newtheorem{theo}{Theorem}[section]
\newtheorem{coro}[theo]{Corollary}

\newtheorem{lemm}[theo]{Lemma}

\newtheorem{prop}[theo]{Proposition}

\theoremstyle{definition}
\newtheorem{defi}[theo]{Definition}

\newtheorem{obse}[theo]{Observation}
\newtheorem{rema}[theo]{Remark}

\newtheorem{exam}[theo]{Example}

\newtheorem{ques}[theo]{Question}

\DeclareSymbolFontAlphabet{\scr}{rsfs}

\newcommand{\FF}{\mathbb{F}}

\newcommand{\NN}{\mathbb{N}}
\newcommand{\PP}{\mathbb{P}}

\newcommand{\ZZ}{\mathbb{Z}}
\renewcommand{\AA}{\mathbb{A}}

\newcommand{\m}{\mathfrak{m}}
\newcommand{\n}{\mathfrak{n}}
\newcommand{\p}{\mathfrak{p}}

\DeclareMathOperator{\colim}{colim}

\DeclareMathOperator{\chain}{chain}

\title{Non-reduced valuation rings and descent for smooth blowup squares}
\author{Shane Kelly}
\date{\today}

\begin{document}

\begin{abstract}
We consider a class of non-reduced valuation rings, known in the literature as \emph{chain rings}. We observe that the Grothendieck topology generated by the Zariski topology and smooth blowup squares is exactly the topology which has chain rings for its local rings, and that sheaves for this topology are \emph{not} characterised by excision for smooth blowup squares.
\end{abstract}

\maketitle

This is collated from notes from early 2020. The motivation was to find a cdh topology which can see nilpotents. The idea was the following. The cdh topos is the classifying topos for hensel valuation rings. %In fact, there is a correspondence between certain topologies and certain classes of rings, Prop.\ref{prop:bijection}. 
A finitary topology on affine schemes of finite presentation can see nilpotents if and only if its category of sheaves is the classifying topos for a class of rings which are not all reduced, Obs.\ref{obse:red}. So one should look for a nice class of rings related to valuation rings but which are not reduced. Chain rings are such a class, but the associated topology was abandoned because its sheaves are not characterised by excision for smooth blowup squares, Prop.\ref{prop:intro3}.

Upon a suggestion of Shuji Saito, this project turned into the procdh topology.

\section{Summary}

We summarise the main points here.

\begin{defi}
A \emph{chain ring} is a ring whose poset of ideals is totally ordered. Let $S$ be a qcqs scheme and $\Sch_S$ is the category of $S$-schemes of finite presentation. The \emph{chain topology} on $\Sch_S$ has as coverings those families $\{Y_i \to X\}_{i \in I}$ such that 
 \[ \amalg \hom(\Spec(R), Y_i) \to \hom(\Spec(R), X) \]
is surjective for every chain ring $R$.
\end{defi}

Every valuation ring is a chain ring and every reduced chain ring is a valuation ring, Lem.\ref{lemm:redChainVal}. Every localisation and quotient of a chain ring is a chain ring. I don't know any examples of chain rings which are not quotients of valuation rings. There are procdh local rings which are not chain rings\footnote{E.g., $k[x,y]/\langle x^2, y^2 \rangle$.} and chain rings which are not procdh local rings.\footnote{E.g., $R / \langle xy \rangle$ where $R$ is any valuation ring with three primes $\m \supseteq \p \supseteq (0)$, $x \in \m \setminus \p$ and $y \in \p \setminus (0)$. %
%As a quotient of a valuation ring this is a chain ring, but it is has Krull dimension one and is its own ring of total quotients, so it can't be a procdh local ring. 
%Indeed, consider a $z \in R$ which is nonzero in $R' := R/\langle xy\rangle$. If $z \in \p$, then there is some $n \in \NN_{\geq 1}$ such that $z^{n-1} \not| y$, $z^n | y$ so $xz^{n-1}$ is nonzero in $R'$ but $xz^{n-1}z$ is zero in $R'$, so $z$ is a zero divisor. if $z \in \m \setminus \p$ then there is some $n \in \NN_{\geq 1}$ such that $z^{n-1} \not| x$, $z^n | x$. So $z^{n-1}y$ is nonzero in $R'$ but $z^{n-1}yz$ is zero in $R'$ so $z$ is a zero divisor. So we have shown that all nonzero nonunits of $R'$ are zero divisors, and therefore $Q(R') = R'$.
} 

\begin{prop}[{Proposition~\ref{prop:equivChain}}]
Let $\sU = \{Y_i \to X\}_{i \in I}$ be a family of morphisms in $\Sch_\ZZ$. The following are equivalent. 
\begin{enumerate}
 \item $\sU$ is a covering for the chain topology.

 \item $\sU$ is refinable by a composition of pullbacks of families of the form:
 \begin{enumerate}
  \item Zariski coverings,
  \item the family 
  \[ \{\{0\} \to \AA^2, Bl_{\AA^2} \{0\} \to \AA^2 \}. \] 
  \end{enumerate}
\end{enumerate}
In other words, the chain topology is generated by the Zariski topology, and blowup of the affine plane in the origin.
\end{prop}

\begin{prop}[{Corollary~\ref{coro:autoRH}}] \label{prop:intro3}
There exist chain sheaves $F$ such that
\begin{equation} \label{equa:A2Bl}
\xymatrix{
F(\AA^2) \ar[r] \ar[d] & F(Bl_{\AA^2}\{0\}) \ar[d] \\
F(\{0\}) \ar[r] & F(\PP^1)	
}
\end{equation}
is not a cartesian square. 

In fact, if $τ \leq \Zar$ is a finitary topology on $\Sch_\ZZ$ such that for all sheaves $F$ the square Eq.\eqref{equa:A2Bl} is cartesian, then $τ \leq rh$ and so $F(X) = F(X_{\red})$ for all sheaves $F$ and $X \in \Sch_\ZZ$.
\end{prop}

\begin{obse}[{Observation~\ref{obse:cute}}]
The rh topology on the category $\Aff_\ZZ$ of affine $\ZZ$-schemes of finite presentation is generated by the two families
 \[ \{ \Spec\ \ZZ[x] \to \Spec\ \ZZ[x,y] / (xy), \quad \Spec\ \ZZ[y] \to \Spec\ \ZZ[x,y] / (xy) \} \]
 and 
 \[ \{ \Spec\ \ZZ[x, \tfrac{y}{x}] \to \Spec\ \ZZ[x,y], \quad \Spec\ \ZZ[\tfrac{x}{y}, y] \to \Spec\ \ZZ[x,y] \} \]
in the sense that a family of morphisms $\{U_i \to X\}_{i \in I}$ in $\Aff_\ZZ$ is covering for the rh topology if and only if it is refinable by a composition of pullbacks of the above two families.
\end{obse}

\emph{Conventions and notation.} For a scheme $S$, we write $\Sch_S$ for the category of $S$-schemes of finite presentation, and $\Aff_S$ for the category of (absolutely) affine $S$-schemes of finite presentation. 

We use the SGA conventions for pretopologies, topologies, and covering families. That is, a \emph{pretopology} is a category equipped with a collection of families $\{U_i \to X\}_{i \in I}$ satisfying certain axioms, while a \emph{topology} is a category equipped with a collection of sieves $R \subseteq \hom(-, X)$ satisfying certain axioms. It is an easy exercise to show that a sieve $R \subseteq \hom(-, X)$ is a $τ$-covering sieve if and only if its sheafification $a_\tau R \to a_\tau \hom(-, X)$ is an isomorphism c.f., \cite[Def.II.1.1(T2), Lemm.II.3.1(2), Thm.II.3.4]{SGA41}.

If $τ$ is a topology, a family $\{U_i \to X\}_{i \in I}$ is called a \emph{$τ$-covering family} if the sieve it generates $\cup_{i \in I} image(\hom(-, U_i) \to \hom(-, X))$ is a covering. If $τ_\pi$ is the topology associated to a pretopology $\pi$, then one can show that the $τ_\pi$-covering families are precisely those families which can be refined by a $\pi$-covering family.

\section{Some first order logic}

In this section, specifically Proposition~\ref{prop:bijection}, we recall a relationship between topologies and ring theories, cf.\cite[Chap.X]{MM94} or \cite[Chap.2]{MR06}. The reader familiar with the notions such as classifying topos, first order theory, syntactical site, etc can skip this section.

%Recall that the collection of Grothendieck topologies on a category $C$ has a natural partial ordering, given by $τ \leq τ'$ if every $τ$-sheaf is a $τ'$-sheaf, that is, if $\Shv_τ(C) \subseteq \Shv_{τ'}(C)$. So for example, étale $ \leq $ Zariski on $\Sch_\ZZ$.

\begin{rema} \label{rema:broad}
In broad strokes, the equivalence in Proposition~\ref{prop:bijection} is the following. A ring  (= a set equipped with two binary relations $+$ and $\cdot$ and two elements $0, 1$ satisfying a list of axioms) of the form 
\begin{equation} \label{equa:polyGen}
\ZZ[x_1, \dots, x_n] / \langle f_1, \dots, f_c \rangle	
\end{equation}
corresponds to a sentence (= a list of characters) of the form
\begin{equation} \label{equa:genSent}
\mathtt{f_1(x_1, \dots, x_n) = 0\  \wedge\ \dots\ \wedge\ f_c(x_1, \dots, x_n) = 0},
\end{equation}
%Rings of the form Eq.\eqref{equa:polyGen} correspond to objects in $\Aff_{\ZZ}$ and then a general ring $R$ induces a limit preserving functor $\hom(\Spec R, -): \Aff_{\ZZ} \to \Set$. There is a sensible way to turn sentences of the form Eq.\eqref{equa:genSent} into a category $\sB$, cf.\cite[\S X.5]{MM94}, and our ring $R$ induces a functor which sends the $\phi$ from Eq.\eqref{equa:genSent} to the set $\{a \in R^n\ |\ \phi(a) \} \subseteq R^n$ where $\phi(a)$ means substitute $a_i$ for $x_i$ and interpret the resulting list of characters as a true/false valued statement in the obvious way. 
%
cf.\cite[\S X.5]{MM94}. Given a topology $τ$ on $\Aff_{\ZZ}$, we associate the class $\sR_τ$ of those rings $R$ such that for every $τ$-covering family 
\begin{equation} \label{equa:Fam}
 \{\Spec R_i \to \Spec R_0\}_{i = 1, \dots, m}	
\end{equation}
the morphism of sets
\begin{equation} \label{equa:RRR}
\hom(R_1, R) \sqcup \dots \sqcup \hom(R_n, R) \to \hom(R_0, R)
\end{equation}
is surjective. Conversely, given a class of rings $\sR$, we make a topology $τ_\sR$ whose covering families are those families Eq.\eqref{equa:Fam} such that for every $R \in \sR$, the morphism of sets Eq.\eqref{equa:RRR} is surjective.
From the logic perspective, surjectivity of Eq.\eqref{equa:RRR} is formalised in the sentence 
\begin{equation} \label{equa:senaabb}
\mathtt{\forall a; \phi_0(a),\ \exists b_1;\phi_1(a, b_1)\  \vee\  \dots\  \vee\  \exists b_n;\phi_n(a, b_n)}.
\end{equation}
Here the sentences $\phi_i$ correspond to the $R_0$-algebras $R_i$, and $\phi(a) = \phi(a_1, \dots, a_n)$ means substitution of elements $a_i \in R$ of the ring into the variables $x_i$, etc, in the obvious way. %Cf.Example~\ref{exam:sentences}.
In this way we obtain a correspondence:
\begin{center}
\fbox{
\begin{tabular}{ccc}
Covering families Eq.\eqref{equa:Fam} for $τ$  & $\leftrightarrow$ & Axioms Eq.\eqref{equa:senaabb} for $\sR$
\end{tabular}
}
\end{center}
Note that different collections of families can generate the same topology, just as different sets of axioms can define the same class of rings. This correspondence is a special case of the correspondence between coherent theories and their classifying topoi, cf.\cite[pg.49 and Thm.9.1.1]{MR06}.
\end{rema}

Now let us be more precise. In order for the operations $τ \mapsto \sR_τ$ and $\sR \mapsto τ_{\sR}$ in Remark~\ref{rema:broad} to be mutually inverse, we need some restriction on which topologies and which classes we consider. On the one hand, our topologies must have enough points. This is assured for finitary topologies by Deligne's completeness theorem, Thm.\ref{theo:deligne}.\footnote{More generally, if $\Sch_S$ has countably many objects (e.g., $S = \Spec(\ZZ)$), and the topology is generated by countably many coverings, each  consisting of countably many morphisms, then \cite[Thm.6.2.4]{MR06} assures that $\Shv(\Sch_S)$ will have enough points. In this case, we should make a more general version of Definition~\ref{defi:coheSent} where countably many $\vee$ are allowed. Such a sentence could be called \emph{geometric}.} On the other hand, our ring axioms are required to have a very specific form. As such, it is  convenient for us to make the following nonstandard definition.
 
\begin{defi} \label{defi:coheSent}
Given a ring $R$, we will say that a sentence $\phi$ is \emph{coherent} if it is of the form\footnote{The semicolons should be read as ``such that''.}
\begin{align*}
\mathtt{
\forall (a_1, \dots, a_n) \in R^n;\quad 
f_1(a) = 0\ \wedge\ 
}&\mathtt{
\dots\ \wedge f_I(a) = 0, 
}\\\mathtt{
\exists (b_{11}, \dots, b_{1m_1}) \in R^{m_1};\quad 
g_{11}(a, b) = 0\ \wedge\ 
}&\mathtt{
\dots\ \wedge\ g_{1K_1}(a, b) = 0 
}\\\mathtt{
\vee \quad
\exists (b_{21}, \dots, b_{1m_2}) \in R^{m_2};\quad 
g_{21}(a, b) = 0\ \wedge\ 
}&\mathtt{
\dots\ \wedge\ g_{1K_2}(a, b) = 0
}\\\mathtt{
\vdots 
}\\\mathtt{
\vee \quad
 \exists (b_{J1}, \dots, b_{1m_J}) \in R^{m_J};\quad 
g_{J1}(a, b) = 0\ \wedge\ 
}&\mathtt{
\dots\ \wedge\ g_{JK_J}(a, b) = 0
}
\end{align*}
for some polynomials $f_i \in \ZZ[x_1, \dots, x_n], g_{jk} \in \ZZ[x_1, \dots, x_n, y_1, \dots, y_{m_j}]$ with $i \in \{1, \dots, I \}$, $j \in \{ 1, \dots, J\}$, $k \in \{1, \dots, K_J\}$.
\end{defi}

\begin{rema}
The term \emph{coherent} is in the same sense as a coherent sheaf of $\OO_X$-modules and refers to the fact that the sentence $\phi$ is a list of \emph{finitely many} symbols, cf.\cite[pg.49]{MR06}. The term \emph{first order} in the title of this section refers to the fact that our variables will be interpreted as elements of a ring as opposed to, say, subsets of a rings. So for example we can't formalise ascending or descending chain conditions on ideals in this setup, or indeed, formalise the notion of ideals at all.
\end{rema}

\begin{defi} \label{defi:coherentRing}
Say that a class of rings $\sR$ is \emph{coherent} if it can be defined by a set of coherent sentences. That is, there exists a set $\{\phi_γ\}_{γ \in \Gamma}$ of coherent sentences such that a ring belongs to $\sR$ if and only if every sentence $\phi_γ$ holds in $R$.
\end{defi}

%\cite[Def.D.1.1.3]{elephant}.

\begin{exam} \label{exam:sentences} \ 
\begin{enumerate}
 \item The class of local rings is coherent: A ring $R$ is local if and only if for every element $a \in R$, either $a$ or $1-a$ is a unit. That is, 
 \[ \mathtt{\forall a \in R,\qquad  \exists b \in R;\   ab -1 = 0\quad  \vee \quad  \exists c \in R; (1-a)c - 1 = 0.} \]
Here we have $g_{1}(x,y) = xy - 1$, $g_{2}(x,z) = (1-x)z-1$ (and there are no $f_i$, that is, $I = 0$).

 \item The class of integral domains is coherent:
 \[ \mathtt{ \forall a, b \in R; ab = 0, \qquad a = 0, \quad  \vee\quad  b = 0.} \]
 Here we have $f(x_1, x_2) = x_1x_2$, $g_{1}(x_1, x_2) = x_1$, $g_2(x_1, x_2) = x_2$. There are no $y$, that is, $m_1, m_2 = 0$.

 \item The class of valuation rings is coherent: A ring $R$ is a valuation ring if and only if it is an integral domain, and for all $a, b \in R$ we have $a|b$ or $b|a$. That is, 
 \[ \mathtt{ \forall a, b \in R, \qquad \exists c \in R; ac = b\quad \vee \quad\exists d \in R; a = bd.} \]

 \item The class of henselian local rings is coherent, but writing a set of defining sentences is not pleasant. The reader might take this as a challenge.
% One definition is: a henselian local ring is a local ring such that for every monic with $a_1 \in R^*$ and $a_0 \in \m$ has a root in $\m$. 
% \[ \forall a_0, \dots, a_{n-1} b; a_1b = 1;
% \exists b; a_0b = 1
% \vee \exists 
% \]

% \item The class of fields is coherent: $\forall a \in R; a = 0 \vee \exists b \in R; ab = 1$.

% \item The class of algebraically closed fields is coherent, but we need countably many $\phi$. In addition to the field axiom, for each $n$, we can use $\forall a_0, a_1, \dots, a_{n-1} \in R^n, \exists b \in R; b^n + \sum_{i = 0}^{n-1} a_ib^i  = 0$.

 \item For a fixed $n \in \NN$, the class of rings of characteristic $n$ is coherent, since they are characterised by: $\mathtt{\forall a \in R, n a = 0}$. Here, $g(y) = n y$ and there are no $f$. On the other hand, the class of rings of finite characteristic is not coherent, since we would need infinitely many $\vee$'s (this class could be called ``geometric'').

 \item It follows directly from the Definition~\ref{defi:coheSent} and Definition~\ref{defi:coherentRing} that any class of coherent rings is closed under filtered colimits. That is, if $\sR$ is a coherent class of rings and $(R_λ)_{λ \in \Lambda}$ is a filtered system of rings in $\sR$, then $\colim R_λ$ is also in $\sR$. Consequently, the class of Noetherian rings is not coherent.

% \item The previous point shows that the class of discrete valuation rings is also not coherent. However, the class of rank $\leq 1$ valuation rings is coherent.

 \item The class of $w$-local rings in the sense of \cite{BS13} is not coherent. To see this, note that the topology on $\Sch_\ZZ$ associated to the class of $w$-local rings is the topology from Example~\ref{exam:affineTopologies}\eqref{exam:affineTopologies:silly}, but the local rings $R$ of this topology do not have to satisfy: the set of closed points of $\Spec(R)$ is closed.
\end{enumerate}
\end{exam}

We make one more definition.

\begin{defi}
A topology on $\Sch_\ZZ$ will be called \emph{affine} if it induces an equivalence of categories $\Shv(\Aff_\ZZ) \cong \Shv(\Sch_\ZZ)$ when we equip $\Aff_\ZZ$ with the induced topology.\footnote{A family $\{U_i \to X\}_{i \in I}$ in $\Aff_\ZZ$ is a covering family if and only if it is a covering family in $\Sch_\ZZ$, \cite[Cor.III.3.3]{SGA41}.} Equivalently, every covering is refinable by one of the form $\{\Spec(A_i) \to X\}_{i \in I}$.
%{\color{red} I checked this. Is it worth writing down? Surely it's already written down somewhere}
\end{defi}

\begin{exam} \label{exam:affineTopologies}
Any topology finer than the Zariski topology is affine. On the other hand, there are affine topologies which are not finer than the Zariski topology.
\begin{enumerate}
 \item \label{exam:affineTopologies:silly} A silly example is the topology generated by singletons $\{Y \to X\}$ such that $Y \in \Sch_\ZZ$ is a surjective disjoint union of open affines of $X$. Sheaves for this topology do not necessarily satisfy $F(X \sqcup X') = F(X) \times F(X')$.%
%\footnote{The $w$-local rings from \cite{BS13} are local for this topology (but not all local rings for this topology are $w$-local since the set of closed points is not necessarily closed).}
%Choose an ordering on the set of primes and take the colimit of $\ZZ_{(p_1)} \times \dots \times \ZZ_{(p_n)} \times \ZZ[p_1^{-1}, \dots, p_n^{-1}]$ over $n$. The set of closed points is not closed.
%the defining sentence is: for all $f, g \in R$ with $\langle a, b\rangle = R$ there exists $e$ with $e^2 = e$ and $f \in (A/e)^*$, $g \in (A/(1-e))^*$.
%more completely, for all $a, b, f, g \in R$ with $af + bg - 1 = 0$, there exists $e, x, y, n, m$ with $e(1-e) = 0$, $xf - 1 + en = 0$ and $yg - 1 + (1-e)m = 0$.

 \item Another silly example is the topology generated by the families $\{\Spec(A) \to X\}_{(\Aff_{\ZZ})_{/X}}$ consisting of the set of \emph{every} morphism from an affine scheme towards $X$. Sheaves for this topology are precisely those presheaves on $\Sch_\ZZ$ which are right Kan extended from $\Aff_{\ZZ}$.

  \item A less silly example is the topology generated by surjective proper morphisms, and $\AA^n$-bundles (for all $n \geq 0$). Indeed, by Nagata compactification and Chow's lemma, any $X$ admits a surjective proper morphism $Y \to X$ with $Y$ quasi-projective. Then by Jouanolou's trick, $Y$ admits a surjective $\AA^n$-bundle $E \to Y$ for some $n$ such that $E$ is affine.%
\end{enumerate}
\end{exam}

We introduce the above terminology to state the following proposition. Recall that a topology is \emph{finitary} if every covering family contains a finite subfamily which is also covering.

\begin{prop} \label{prop:bijection}
There is a bijection of posets between:
\begin{enumerate}
 \item affine finitary topologies $τ$ on $\Sch_\ZZ$, and 
 \item coherent classes of rings.
% Classes of rings $\sP$ which can be defined by coherent sentences. That is, classes $\sP$ for which there exists a set of finite tuples of polynomials 
% \[ \left \{ (f_γ, g_{γ1}, \dots, g_{γ λ_γ})\ \middle | \begin{array}{c} f_γ \in \ZZ[x_1, \dots, x_{n_γ}], \\ g_{γd} \in \ZZ[x_1, \dots, x_{n_γ}, y_1, \dots, y_{m_λ}] \end{array} \right \}_{γ \in \Gamma} \] 
%such that a ring $R$ is in $\sP$ if and only if 
%\[ \forall (a_i) \textrm{ s.t. }f_γ(a_i) = 0;\quad  \exists (b_{ij}), \textrm{ s.t. } g_{γ1}(b_{ij}) = 0, \dots, g_{γ1}(b_{ij}) = 0. \] 
%for all $γ \in \Gamma$.
\end{enumerate}
The bijection sends a class $\sR$ to the topology $τ_{\sR}$ whose coverings are those families $\{U_i \to X\}_{i \in I}$ such that for all $R$ in $\sR$ the morphism of sets
\begin{equation} \label{equa:lifting}
\amalg_{i \in I} \hom(\Spec(R), U_i) \to \hom(\Spec(R), X)
\end{equation}
is surjective. It sends a topology $τ$ to the class $\sR_τ$ of rings $R$ such that for all $τ$-covering families \eqref{equa:lifting} is surjective.
\end{prop}

\newcommand{\local}[2]{(#1 \perp #2)}

\begin{defi}
Given a ring $R$ and a family $\sU = \{U_i \to X\}_{i \in I}$ let's name the condition:
\begin{enumerate}
 \item[{$\local{R}{\sU}$}] The morphism $\amalg_{i \in I} \hom(\Spec(R), U_i) \to \hom(\Spec(R), X)$ is surjective.
\end{enumerate}
so we can refer to it more easily.
\end{defi}

Proposition~\ref{prop:bijection} is a formal consequence of Deligne's completeness theorem.

\begin{theo}[Deligne, {\cite[00YQ]{stacks-project}}] \label{theo:deligne}
Suppose $C$ is a small category admitting finite limits equipped with a finitary topology $τ$. Then the collection of fibre functors $\phi: \Shv_τ(C) \to \Set$ is conservative. 

That is, a morphism $f: F \to G$ of sheaves is an isomorphism if and only if $\phi(f)$ is an isomorphism 
for every $\phi$ which preserves colimits and finite limits.
\end{theo}

\begin{rema} \label{rema:DeligneCov}
There is an equivalent statement of Deligne's theorem: Suppose $C$ is a small category admitting finite limits equipped with a finitary topology $τ$. Then a family $\{Y_i \to X\}_{i \in I}$ is a covering family if and only if 
\[ \amalg_{i \in I} \phi(a_τ\hom(-, Y_i)) \to \phi(a_τ\hom(-, X)) \]
is a surjective morphism of sets for every $\phi$ which preserves colimits and finite limits, \cite[Prop.IV.6.5]{SGA41}.
%a morphism in $\Shv_τ(C)$ is an effective epimorphism
%\footnote{A morphism $f: Y \to X$ in a category with fibre products is an \emph{effective epimorphism} if $\coeq(Y \times_X Y \rightrightarrows Y) \to X$ is an isomorphism.} 
%if and only if its image under every fibre functor is an effective epimorphism. In fact, this is an equivalent statement. 
%
%Indeed, suppose that the potentially weaker version holds, and consider a morphism $f: Y \to X$ such that $\phi(f)$ is an isomorphism for every fibre functor $\phi$. Then $\phi(Y) \to \phi(Y) \times_{\phi(X)} \phi(Y) = \phi(Y \times_X Y)$ is also an isomorphism, so it is an effective epimorphism, so $Y \to Y \times_X Y$ is an effective epimorphism, i.e., $\coeq(Y \times_{Y \times_X Y}Y \rightrightarrows Y) = Y \times_X Y$. But $Y \times_{Y \times_X Y}Y = Y$ (they represent the same functor), so $Y \times_X Y = \coeq(Y \rightrightarrows Y) = Y$. Now since $\phi(f)$ is an effective epimorphism for every fibre functor, $f$ is an effective epimorphism, so $X = \eq(Y \times_X Y \rightrightarrows Y) = \eq(Y \rightrightarrows Y) = Y$.
\end{rema}

We will also use the following well-known formal fact, explained in \cite{GK15}.

\begin{lemm} \label{lemm:fibreRing}
Suppose that $τ$ is an affine topology on $\Sch_\ZZ$. Then there is a canonical bijection
\[ \sR_τ \cong \left \{ \begin{array}{c} \textrm{ fibre functors } \\ \textrm{ of } \Shv_τ(\Sch_\ZZ) \end{array} \right \}  \] 
A ring $R \in \sR_τ$ corresponds to the functor 
\[ \phi_R: F \mapsto \underset{{\Spec(R) \to X}}{\colim} F(X) \]
where the colimit is over all $X \in \Sch_\ZZ$. 
\end{lemm}

\begin{rema} \label{rema:pflim}
Note that since objects in $\Sch_\ZZ$ are finite presentation, and we can write $\Spec\ R$ as a filtered limit of affine $\ZZ$-schemes of finite presentation, for $Y \in \Sch_\ZZ$ we have, \cite[Prop.8.13.1]{EGAIV3}, 
\[ \phi_R(\hom(-, Y)) = \hom(\Spec(R), Y). \]
\end{rema}

\begin{proof}[Proof of Proposition~\ref{prop:bijection}.]
First we show that $\sR = \sR_{τ_\sR}$ for every coherent class $\sR$. Let $\sR$ be a coherent class and choose a set of defining sentences $\{\phi_γ\}_{γ \in \Gamma}$. 
To a coherent sentence $\phi$ as in Definition~\ref{defi:coheSent}, we associate the family
\[ \sU_\phi = \biggl \{ Y_{\phi k} := \underset{}{\Spec\ \tfrac{\ZZ[x_1, \dots, x_n, y_1, \dots, y_{m_j}]}{\langle f_i,  g_{jk} \rangle}} \to X_\phi := \Spec\ \tfrac{\ZZ[x_1, \dots, x_n]}{\langle f_i \rangle} \biggr \}_{j \in J}. \]
Then we have 
\[ R \in \sR \iff \forall γ \in \Gamma; \local{R}{\sU_{\phi_γ}}. \]
Now note that one can concretely describe the coverings of $τ_\sR$. They are those families which are refinable by a composition of pullbacks of a families of the form $\sU_{\phi_γ}$; $γ \in \Gamma$. 
%Indeed, every family of the form $\sU_{\phi_γ}$ is in $τ_{\sR}$, and the covering families in $τ_{\sR}$ are stable under pullback, composition, and corefinement (by the axioms of a Grothendieck topology). 
The condition $\local{R}{\sU}$ is preserved by pullback, composition, and corefinement in $\sU$, in the sense that if $\local{R}{\{U_i {\to} X\}}$ holds then $\local{R}{\{Y \times_X U_i \to Y\}}$ holds, etc. 
So it follows that (for a fixed ring $R$) we have 
\[ \forall γ \in \Gamma; \local{R}{\sU_{\phi_γ}} \iff \forall \sU \in τ_{\sR}; \local{R}{\sU}, \] where $\sU \in τ$ means $\sU$ is a covering family for $τ$. Combining the two display equations gives $\sR = \sR_{τ_{\sR}}$, since by definition, 
\[ R \in \sR_{τ_{\sR}} \iff \forall \sU \in τ_{\sR}; \local{R}{\sU}. \]

Now we show $τ = τ_{\sR_{τ}}$. As above, it follows immediately from the definitions that $τ_{\sR_τ} \leq τ$. Indeed, $\sU \in τ \implies \forall R \in \sR_{τ}; \local{R}{\sU} \iff \sU \in τ_{\sR_{τ}}$.

For the converse, suppose $\sU = \{U_i \to X\}_{i \in I}$ is a $τ_{\sR_τ}$-covering family. By definition, this means that $\local{R}{\sU}$ holds for every $R$ in $\sR_τ$. That is, for every $R$ in $\sR_τ$ the morphism
\[ \amalg \phi_R(\hom(-, U_i)) \to \phi_R(\hom(-, X)) \]
is surjective, cf.Remark~\ref{rema:pflim}. By Lemma~\ref{lemm:fibreRing}, this means precisely that $\amalg \phi a_τ \hom(-, U_i) \to \phi a_τ \hom(-, X)$ is surjective for every fibre functor of $\Shv_τ(\Sch_\ZZ)$. By Deligne's theorem in the form of Remark~\ref{rema:DeligneCov}, it follows that $\sU$ is a $τ$-covering family.
\end{proof}

\begin{rema}
Looking at the proof of Proposition~\ref{prop:bijection} one notes that $\sR = \sR_{τ_\sR}$ also holds for \emph{geometric} classes, where by geometric we mean that we allow infinitely many $\vee$ in the sentences in Definition~\ref{defi:coheSent} (i.e., indexed by a small set). On the other hand, the $\tau = \tau_{\sR_{\tau}}$ direction only uses the existence of a conservative family of fibre functors. So the same proof gives the following more general statement.

There is a bijection of posets between:
\begin{enumerate}
 \item affine topologies $τ$ on $\Sch_\ZZ$ admitting a conservative family of fibre functors, and
 \item geometric classes of rings.
\end{enumerate}
If one starts with an arbitrary affine topology $τ$ on $\Sch_\ZZ$ then $τ_{\sR_{τ}}$ is the coarsest topology whose family of fibre functors is conservative such that $τ \leq τ_{\sR_{τ}}$. Conversely, if we start with an arbitrary class of rings $\sR$ then $\sR_{τ_{\sR}} \supseteq \sR$ is the class of those rings which satisfy every geometric sentence satisfied by all rings in $\sR$.
\end{rema}

We can apply Proposition~\ref{prop:bijection} to make the following cute observations. 

\begin{obse}\ \label{obse:cute}
\begin{enumerate}
 \item \label{obse:cute:Zar} The Zariski topology on $\Aff_\ZZ$ is generated by the family 
 \[ \{ \AA^1 \setminus \{0\} \to \AA^1, \AA^1 \setminus \{1\} \to \AA^1 \} \]
in the sense that every Zariski covering is refinable by a composition of pullbacks of this family. One can also show this directly without too much difficult (exercise). 

Note, that this covering does not generate the Zariski topology on $\Sch_\ZZ$, since there are no surjective morphisms from proper schemes towards $\AA^1$.%
\footnote{Incidentally, the Zariski topology on $\Sch_\ZZ$ is generated by $\{\PP^1 \setminus \{0\} \to \PP^1, \PP^1 \setminus \{\infty\} \to \PP^1\}$ (exercise). 
%Indeed, if this is a covering, then so is $\{U_{\sL, s} \to X, U_{\sL, t} \to X\}$ for any line bundle $\sL$ on $X$ where $U_{\sL, s} = \{ x\ |\ s$ is nonzero in $\OO_{X,x}\}$ such that $U_{\sL, s} \cup U_{\sL, t} = X$ (this is the universal property of $\PP^1$). Then by induction, families of the form $\{U_{\sL,s_i} \to X\}_{i = 0, \dots, N}$ are also coverings. This means that for any graded ring $S$ generated by $S_1$, the all usual Zariski coverings from the Proj construction are coverings. This means that for any quasi-projective scheme, all usual Zariski coverings are coverings. Finally, to get to all schemes, use Chow's lemma and the fact that proper morphisms are topological epimorphisms, i.e., the topology on the target is the quotient topology.
}

 \item The rh-topology on $\Aff_\ZZ$ is generated by the two families
 \[ \{ \Spec\ \ZZ[x] \to \Spec\ \ZZ[x,y] / (xy), \quad \Spec\ \ZZ[y] \to \Spec\ \ZZ[x,y] / (xy) \} \]
 \[ \{ \Spec\ \ZZ[x, \tfrac{y}{x}] \to \Spec\ \ZZ[x,y], \quad \Spec\ \ZZ[\tfrac{x}{y}, y] \to \Spec\ \ZZ[x,y] \} \]

% \item the rarc-topology, \cite[Def.3.1.1]{EHIK}, on $\Sch_\ZZ$ is generated by the rh-topology and the countable family of $\Spec\ \ZZ[x,y]$-schemes
%\[ \{ \Spec\ \ZZ[x,y] / (xy), \qquad \Spec\ \ZZ[x,y,z,w] / (x^n = zy, y^m = wz) \}_{n, m \in \NN} \]
\end{enumerate}
\end{obse}

Now a more serious observation.

\begin{obse} \label{obse:red}
A finitary affine topology $τ$ on $\Sch_\ZZ$ has $F(X) = F(X_{\red})$ for every sheaf $F$ and $X \in \Sch_\ZZ$ if and only if every $R \in R_{τ}$ is reduced.
\end{obse}

Hence, if we want a ``cdh-topology'' which can see nilpotents, the obvious choice is $τ_\sR$ for some class of nonreduced rings $\sR$ containing henselian valuation rings.

%\begin{rema}
%It is not so difficult to give a direct proof of Observation~\ref{obse:cute}\eqref{obse:cute:Zar} by hand, without using Deligne's theorem.
%\end{rema}

\section{Non-reduced valuation rings; definition}

Recall the following.

\begin{lemm} \label{lemm:valuRingDef}
Suppose that $R$ is an integral domain. The following conditions are equivalent.
\begin{enumerate}
 \item The poset of ideals of $R$ is totally ordered.
 \item Every finitely generated ideal of $R$ is principal, and $R$ is local.
 \item For every $x, y \in R$ either $x$ divides $y$ or $y$ divides $x$.
\end{enumerate}
\end{lemm}

\begin{proof}
$(1) \Rightarrow (2)$. For locality, notice that since the ideals are totally ordered, the union of the non-proper ideals is again a non-proper ideal, necessarily maximal.
For principality, suppose that $I$ is an ideal with $n$ generators $f_1, \dots, f_n$. Since the poset of ideals is totally ordered, either $(f_n) \subseteq (f_{n-1})$ or $(f_{n-1}) \subseteq (f_n)$. Hence, $I$ can be generated by $n-1$ elements. By induction on $n$, the ideal $I$ is principal.

$(2) \Rightarrow (3)$. Suppose $I$, $J$ are two ideals, principal by assumption, so $I = (x)$ and $J = (y)$. Again by assumption $(x, y)$ is principal, so there exists $z$ with $(x, y) = (z)$. That is, there exist $a, b, c, d$ with $z = ax + by$ and $x = cz$, $y = dz$. Substituting the former into the latter gives $(1-ca)x = cby$. Since $R$ is local, either $ca$ is a unit, or $1-ca$ is a unit. If $ca$ is a unit then $c$ is a unit so $(y) \subseteq (x, y) = (z) = (cz) = (x)$. That is, $x$ divides $y$. If $1-ca$ is a unit then $(x) = ((1-ca)x) = (cby) \subseteq (y)$. That is, $y$ divides $x$.

$(3) \Rightarrow (1)$. If there exist ideals $I$, $J$ with $I \not\subseteq J$ and $J \not\subseteq I$, then there exist $x, y$ with $x \in I$, $y \in J$, $x \notin J$, $y \notin I$. But then we cannot have $x$ divides $y$ or $y$ divides $x$.
\end{proof}

\begin{defi}
An integral domain satisfying the  equivalent conditions of Lemma~\ref{lemm:valuRingDef} is known as a \emph{valuation ring}. Noetherian valuation rings are called \emph{discrete valuation rings}.
\end{defi}

\begin{rema}
The term ``valuation'' comes from the fact that the abelian group $\Gamma := (\Frac R)^* / R^*$ has a canonical total ordering, induced by the divisibility relation. The canonical map $v: R \setminus \{0\} \to \Gamma$ is called the \emph{valuation} of $R$. If $R$ is Noetherian we necessarily have $\Gamma \cong (\ZZ, +)$ via a choice of generator for the maximal ideal $\m$.
\end{rema}

Now observe that in the proof of Lemma~\ref{lemm:valuRingDef}, the assumption that $R$ was an integral domain was never used.

\begin{defi}
A ring (not necessarily an integral domain) satisfying the  equivalent conditions of Lemma~\ref{lemm:valuRingDef} is known as a \emph{chain ring}.
\end{defi}

\begin{exam}\ 
\begin{enumerate}
 \item Every valuation ring is a chain ring.
 \item Every localisation of a chain ring is a chain ring.
 \item Every quotient of a chain ring is a chain ring.
 \item Conversely, one can show that every Noetherian chain ring is a field, or a quotient of a discrete valuation ring, Prop.\ref{prop:NoethClass}.
 \item It is also possible to show that if $R$ is a chain ring over $\FF_p$ and the Frobenius is surjective, then $R$ is a quotient of a (perfect) valuation ring, Prop.\ref{prop:chainQuot}.
\end{enumerate}
\end{exam}

\begin{rema}
We do not know examples of any chain rings which are not quotients of valuation rings. One may conjecture that every chain ring is a quotient of a valuation ring.
\end{rema}

\begin{ques}
Is every chain ring the quotient of a valuation ring?
\end{ques}

Clearly every chain ring which is an integral domain is a valuation ring (by definition). But in fact it suffices to be reduced. This follows from the more general fact that the minimal prime is the nilradical.

\begin{lemm} \label{lemm:redChainVal}
Suppose that $R$ is a chain ring. Then $R$ has a unique minimal prime, and it is the nilradical of $R$. Consequently, if $R$ is a reduced chain ring, then $R$ is a valuation ring.
\end{lemm}

\begin{rema}
In particular, this shows that chain rings $R$ have the curious property that the nilradical $Nil(R)$ is the unique minimal prime ideal $\n$, and the set of zero divisors $ZD(R)$ is the unique maximal ideal $\m$.
\[ Nil(R) = \n, \qquad ZD(R) = \m. \]
\end{rema}

\begin{proof}
First recall that the nilradical $\n$ is the intersection of all prime ideals. Since the ideals in a chain ring are totally ordered, we are reduced to showing that this intersection is again a prime ideal. Suppose that $ab \in \n$. For all primes $\p$, either $a \in \p$ or $b \in \p$. If both $a, b$ are in all primes, then $a, b \in \n$. If there is a prime which does not contain, say $a$, then all smaller primes do not contain $a$, so all smaller primes must contain $b$. So $b = \cap \p = \n$.

For the ``consequently'' part, if $R$ is a reduced chain ring, $\n = (0)$, i.e., $(0)$ is prime, so $R$ is an integral domain.
\end{proof}

The following lemma may or may not be useful.

\begin{lemm}
Let $A \neq 0$ be a zero dimensional chain ring. Then $A$ is a colimit (even union) of local Artin rings (but not necessarily Artin chain rings).
\end{lemm}

\begin{proof}
Consider any finite set of elements $a_1, \dots, a_n \subseteq A$ and the induced morphism $\ZZ[t_1, \dots, t_n] \to A$. Let $\p \subseteq \ZZ[t_1, \dots, t_n]$ be the preimage of the unique prime of $A$. Let $\OO = \ZZ[t_1, \dots, t_n]_\p$ be the localisation at $\p$ so we now have a local morphism of local rings $\OO \to A$. Let $\m_\OO$ be the maximal ideal of $\OO$, and $f_1, \dots, f_m$ a finite set of generators. They are sent inside the maximal ideal $\m_A$ of $A$, as $\OO \to A$ is a local homomorphism. But $\m_A$ consists of nilpotents. So there is some integer $N$ such that $\OO \to A$ factors as $\OO / \m_\OO^N \to A$. So now we have a morphism from an Artin ring whose image contains all the original $a_1, \dots, a_n$. Since $R$ is non-zero, the kernel of $\OO / \m_\OO^N \to A$ is contained in $\m_\OO$, and quotienting by this, we obtain an injective morphism $\OO / ker \subseteq A$ from an Artin ring $A_0 := \OO / ker$.
\end{proof}

\section{Chain rings as local rings}

\begin{defi}
Let $τ_{\chain}$ be the topology such that $\sR_{τ_{\chain}}$ is the class of chain rings. Let's call this the \emph{chain topology}.
\end{defi}

In the spirit of Observation~\ref{obse:cute} we have:

\begin{obse}
The chain ring topology on $\Aff_\ZZ$ is generated by the family  \[ \{ \Spec\ \ZZ[x, \tfrac{y}{x}] \to \Spec\ \ZZ[x,y], \quad \Spec\ \ZZ[\tfrac{x}{y}, y] \to \Spec\ \ZZ[x,y] \}. \]
\end{obse}

In fact, the chain ring topology can also be described in terms of smooth blowup squares.

\begin{prop} \label{prop:equivChain}
Let $\sU = \{Y_i \to X\}_{i \in I}$ be a family of morphisms in $\Sch_\ZZ$. The following are equivalent. 
\begin{enumerate}
 \item For every chain ring $R$ the morphism
 \[ \amalg \hom(\Spec(R), Y_i) \to \hom(\Spec(R), X) \]
 is surjective. That is, $\sU$ is a covering for the chain topology.

 \item $\sU$ is refinable by a composition of pullbacks of families of the form:
 \begin{enumerate}
  \item Zariski coverings,
  \item smooth blowups, i.e, families of the form 
  \[ \{ Z \to X, Bl_{X} Z \to X \}. \] 
  where $Z$ and $X$ are smooth.
  \end{enumerate}

 \item $\sU$ is refinable by a composition of pullbacks of families of the form:
 \begin{enumerate}
  \item Zariski coverings,
  \item the families 
  \[ \{\{0\} \to \AA^n, Bl_{\AA^n} \{0\} \to \AA^n \}, \] 
  for $n \geq 2$.
  \end{enumerate}

 \item $\sU$ is refinable by a composition of pullbacks of families of the form:
 \begin{enumerate}
  \item Zariski coverings,
  \item the family 
  \[ \{\{0\} \to \AA^2, Bl_{\AA^2} \{0\} \to \AA^2 \}. \] 
  \end{enumerate}
\end{enumerate}
\end{prop}

\begin{proof}
Clearly $(1) \iff (4) \implies (3) \implies (2)$, so it suffices to show that the families $\{ Z \to X, Bl_{X} Z \to X \}$ in (2) satisfy (1). Suppose that $R$ is a chain ring and $\Spec(R) \to X$ is any morphism. We will show that $\Spec(R) \to X$ factors through $Bl_X Z$. Let $x \in X$ be the image of the closed point of $\Spec(R)$. Since $X$ and $Z$ are smooth, there exists an open affine neighbourhood $U$ of $x$, and an étale morphism $f: U \to \AA^d$ such that $Z \cap U = f^{-1}\AA^{d-c}$ where $d$ is the dimension of $X$ and $c$ the local codimension of $Z$. Replacing $X$ with $U$, we can assume the morphism is $f: X \to \AA^d$. Since it is étale, it is flat, so $Bl_{X}Z$ is the pullback of $Bl_{\AA^d} \AA^{d-c}$. So in fact we can assume $X = \AA^d$ and $Z = \AA^{d-c}$. But this is the pullback of $Bl_{\AA^{c}} \{0\}$ so we assume $d = c$.

Let $a_1, \dots, a_d$ be the images in $R$ of $x_1, \dots, x_d \in \ZZ[x_1, \dots, x_d]$. 
%If all $a_i$ are zero, then $\Spec(R) \to X$ factors through $Z$. If not, 
Since $R$ is a chain ring, there is one $a_i$ which divides the others. So $\Spec(R) \to \Spec\ \ZZ[x_1, \dots, x_d]$ factors through $\Spec\ \ZZ[\tfrac{x_1}{x_i}, \dots, x_i, \dots, \tfrac{x_d}{x_i}]$. But this is one of the standard opens of the blowup $Bl_X Z$. So $\Spec(R) \to X$ factors through $Bl_X Z$.
\end{proof}

An obvious question is if the chain ring topology fits into Voevodsky's theory of (bounded, complete, regular) cd-structures. It does not. 

\begin{prop} \label{prop:cdStruc}
Suppose that we are given a collection $P$ of cartesian squares
\[ \xymatrix{
B \ar[r] \ar[d] & Y \ar[d] \\
A \ar[r]_i & X
} \]
in $\Sch_\ZZ$ such that $i$ is a categorical monomorphism. Let $τ$ be the topology generated by $P$,\footnote{So $τ$-covering families are those which are refinable by composition of pullbacks of families of the form $\{A \to X, Y \to X\}$ for some square in $P$.}, suppose $τ$ is affine, and let $\sR$ be the class of $τ$-local rings. Then the following are equivalent.
\begin{enumerate}
 \item A presheaf $F$ is a sheaf if and only if 
\[ \xymatrix{
F(B) & \ar[l] F(Y)  \\
F(A) \ar[u] & \ar[l] F(X) \ar[u]
} \]
is cartesian for all pullbacks of squares in $P$.

 \item The image of every square of $P$ in $\Shv_τ(\Sch_\ZZ)$ is cocartesian.
 
 \item For every $R \in \sR$ the square
\[ \xymatrix{
\phi_R(B) \ar[r] \ar[d] & \phi_R(Y) \ar[d] \\
\phi_R(A) \ar[r] & \phi_R(X)
} \]
is a cocartesian square of sets.

 \item For every $R \in \sR$ the morphism
 \[ \phi_R(Y) \setminus \phi_R(B) \to \phi_R(X) \setminus \phi_R(A) \]
 is injective.
\end{enumerate}
\end{prop}

\begin{rema}
Pay attention that the morphism in part (4) is not the morphism $\phi_R(Y \setminus B) \to \phi_R(X \setminus A)$.
\end{rema}

\begin{proof}
(1) $\iff$ (2) is essentially just Yoneda, colimits being universal in topoi, and the factorisation $B \to (A \sqcup Y) \times_X (A \sqcup Y) \rightrightarrows (A \sqcup Y)$.

(2) $\iff$ (3) is the fact discussed above that the $\phi_R$ form a conservative family of fibre functors (consider the morphism $A \sqcup_B Y \to X$ in $\Sch_τ(\Sch_\ZZ)$).

(3) $\iff$ (4) is the fact that a cartesian square of sets with monic horizontal morphisms is cocartesian if and only if the ``complement'' morphism is a monomorphism.
\end{proof}

\begin{coro} \label{coro:autoRH}
Suppose $τ \leq \Zar$ is a finitary topology on $\Sch_\ZZ$ such that the image of 
\[ \xymatrix{
\PP^1 \ar[r] \ar[d] & Bl_{\AA^2} \{0\} \ar[d] \\
\{0\} \ar[r] &  \AA^2
} \]
in $\Shv_τ(\Sch_\ZZ)$ is cocartesian. Then all local rings are valuation rings. In particular, $τ \leq rh$ and all $X_{\red} \to X$ are sent to isomorphisms in $\Shv_τ(\Sch_\ZZ)$.
\end{coro}

\begin{proof}
First consider the cd-structure which is the union of the Zariski cd-structure, and the square above. By Proposition~\ref{prop:equivChain} the local rings of $τ$ are all chain rings. On the other hand, by Proposition~\ref{prop:cdStruc} the morphism\footnote{This is the composition of one of the standard opens $U \subseteq Bl_{\AA^2} \{0\}$ and the canonical morphism $Bl_{\AA^2} \{0\} \to \AA^2$.}
\[ R \times R \to R \times R; \qquad  (x,y) \mapsto (x, xy) \]
is injective when restricted to the preimage of $R \times R \setminus \{(0,0)\}$. In particular, if there exist $x, y \in R$ with $x \neq 0$ and $xy = 0$, then we must have $y = 0$. That is, there are no nonzero zero divisors. A chain ring without nonzero zerodivisors is a valuation ring by Lemma~\ref{lemm:redChainVal}. Hence $τ \leq rh$ by Proposition~\ref{prop:bijection}.
\end{proof}

\section{Chain rings as quotients of valuation rings}

\begin{prop} \label{prop:NoethClass}
Suppose that $A$ is a Noetherian chain ring. Then either:
\begin{enumerate}
 \item $A$ is a field, 
 \item $A$ is a dvr, 
 \item $A = k[[t]] / (t^n)$ for some field $k$ and some integer $n \geq 1$, or 
 \item $A = \Lambda / (p^n)$ for some complete discrete valuation ring $\Lambda$ whose uniformiser is a prime $p \in \ZZ$ and $n \geq 1$.
\end{enumerate}
\end{prop}

\begin{proof}
Let $A$ be a Noetherian chain ring. Let $t$ be a nonzero divisor, and $e$ a zero divisor. We claim that either $t$ is a unit or $e = 0$. This implies that either:
\begin{enumerate}
 \item all nonzero divisors are units, or
 \item $A$ is an integral domain.
\end{enumerate}

Proof that either $t \in A^*$ or $e = 0$. Let $a \in A$ be such that $at = e$ or $t = ae$. We cannot have $t = ae$ because $t$ is not a zero divisor. So $at = e$. The same is true for $t^n$ in place of $t$ for all $n \geq 1$. So we find a sequence of elements $a_1, a_2, a_3, \dots$, such that $a_nt^n = e$. Notice that $a_nt^n = a_{n+1}t^{n+1}$ so $(a_n - a_{n+1}t)t^n = 0$. Since $t$ is not a zero divisor, we find that $a_n = a_{n+1}t$, or in other words, $(a_n) \subseteq (a_{n+1})$. Since $A$ is Noetherian, this sequence of ideals stabilises. So there is some $n$ such that $a_n$ and $a_{n+1}$ differ by a unit. That is, $a_n = u a_n t $ for some unit $u$. Then $a_n(1 - ut) = 0$. If $t$ is not a unit, then $(1 - ut)$ is a unit, so $a_n = 0$, and therefore $e = a_nt^n = 0$.

If $A$ is a Noetherian integral domain, then it is a field or a dvr, so assume that $A$ is not an integral domain. By the above, this implies all nonzero divisors in $A$ are units. Since $A$ is local, this implies that the maximal ideal---the set of non-units---is exactly the set of zero divisors. Using a similar argument to the above we claim that we can show that all zero divisors are nilpotent.

Proof that all zero divisors are nilpotent. Suppose that $xy = 0$, with $x, y \neq 0$, and that for some $n \in \NN$ we have found 
\[ y_{-1} = 0, \quad y_0 = y,\quad  y_1, \dots, y_n \]
such that $y_n = y_{n+1}x$ for all $n = -1, \dots, n{-}1$. Note that by induction we have $y_ix^i = y$ (it may help to think of $y_i$ as $y / x^{i}$). Now we must have $x|y_n$ or $y_n|x$. In the case $n = 0$, without loss of generality we can assume $y = y_1x$ for some $y_1$. In the case $n > 0$, if $x = z y_n$ then $xx^nx = zy_nx^nx = zyx = 0$ and $y^{n+2} = (y_1x)^{n+2} = 0$, so $x$ and $y$ are nilpotent. If not, then $y_n = y_{n+1}x$ and we continue the sequence. So either $x$ and $y$ are nilpotent, or we can find $y_n$ for all $n \in \NN$ with $y_n = y_{n+1}x$. Since $A$ is Noetherian, the sequence of ideals $\dots \subseteq (y_n) \subseteq (y_{n+1}) \subseteq$ stabilises, and we have $y_{n+1} = zy_n$ for some $n$. Since we also have $y_n = y_{n+1}x$ we find that $y_{n+1}(1-zx) = 0$. Since $x$ is a zero divisor, $zx$ cannot be a unit, so $1-zx$ is a unit, and we find that $y_{n+1} = 0$, implying that $y_{n+1}x^{n+1} = y = 0$, a contradiction.

So all zero divisors are nilpotent. We have seen previously that the set of zero divisors is the maximal ideal, so we deduce that $A / \m = A_{red}$ is a field. Or in other words, $A$ is a dimension zero local ring, or in other words, a local Artin ring. In particular, it is complete. The result then follows from Cohen's structure theorem, [Stacks Project, 0323].
\end{proof}

\begin{lemm}
Suppose $A$ is a chain ring in positive characteristic. Then the colimit perfection $A_{\perf} = \colim(A \stackrel{Frob}{\to} A \stackrel{Frob}{\to} \dots)$ is also a chain ring.
\end{lemm}

\begin{proof}
Any two elements $x, y \in A_{\perf}$ are in the image of some $A$, where one divides the other.
\end{proof}

\begin{rema}
The passage $A \mapsto A_{\perf}$ loses all information about nilpotents. Indeed, if $a \in A$ satisfies $a^n = 0$ then there is some $m$ for which $a^{p^m} = 0$. So $A_{\perf}$ is reduced. As we observed above, this implies it is a valuation ring. In fact, it is the colimit perfection of the valuation ring $A_{red}$.
\end{rema}

\begin{prop} \label{prop:chainQuot}
Suppose that $A$ is a chain ring in positive characteristic $p$. The inverse limit perfection
\[ A^\perf = \varprojlim(\dots \stackrel{Frob}{\to} A \stackrel{Frob}{\to} A ) \]
is a valuation ring. In particular, if the Frobenius is surjective, then $A$ is the quotient of a (perfect) valuation ring.
\end{prop}

We isolate the following lemma for readibility.

\begin{lemm} \label{lemm:annKer}
Suppose that $A$ is a chain ring of characteristic $p$, and $xy = 0$. If $y^p \neq 0$, then $y | x$ and $x^p = 0$.
\end{lemm}

\begin{proof}
If $y | x$, then $x = yz$ so $x^p = x^{p-1}yz = x^{p-2}xyz = 0$. If $x|y$, by symmetry $y^p = 0$, a contradiction.
\end{proof}

\begin{proof}
Suppose $a = (\dots, a_2, a_1), b \in A^\perf \subseteq \prod_{\NN} A$ are nonzero elements. That is, we have $a_i, b_i \in A$ and 
\begin{equation} \label{equa:aabb}
a_{i+1}^p = a_i, b_{i+1}^p = b_i
\end{equation}
and 
\begin{equation} \label{equa:anbn}
a_n, b_n \neq 0
\end{equation}
 for some $n$. 
%This means that $a_{n+m}, b_{n+m} \neq 0$ for all $m \geq n$. 
Now if $ab = 0$, then we have $0 = a_{n+1}b_{n+1}$ so by Lemma~\ref{lemm:annKer}, since $0 \neq a_n = a_{n+1}^p$ we must have $b_n = b_{n+1}^p = 0$, contradicting the assumption that $b_n \neq 0$. So $A^\perf$ is an integral domain. 

Next, we show that $A^\perf$ is a chain ring. As above, let $a, b$ be two nonzero elements, such that $a_n, b_n$ are nonzero. Since $A$ is a chain ring, we have either $b_n|a_n$ or $a_n|b_n$. Suppose $b_n|a_n$. We will show that $b|a$ by constructing a $c$ with $a = bc$. 

First, note that we have $b_i | a_i$ for all $i \leq n$ because $b_i = b_{i+1}^p$ and $a_i = a_{i+1}^p$. In the other direction we must also have $b_{n+1}|a_{n+1}$. Indeed, if $b_{n+1} = ca_{n+1}$ for some non-unit $c$, then $b_n = c^pa_n = c^pb_nc_n$ for some $c_n$ so $(1-c^pc_n)b_n = 0$. By assumption $c$ is not a unit, so $c^pc_n$ is not a unit so because $A$ is local, $(1-c^pc_n)$ is a unit, so $b_n = 0$, a contradiction. 

So by induction $b_i | a_i$ for all $i \in \NN$. For each $i$, choose $c_i \in A_i$ such that 
\begin{equation} \label{equa:abc}
a_i = b_ic_i.
\end{equation}
We do not necessarily have $c_{i+1}^p = c_i$, but we claim that upon replacing $c_i$ with $c_i' := c_{i+2}^{p^2}$ for all $i \geq n$, we have both 
\[ a_i = b_ic_i' \quad \textrm{ and } \quad (c_{i+1}')^p = c_i' \]
for all $i \geq n$, and therefore $a = bc$, or in other words, $b | a$ (of course for $i < n$ we set $c_i' = (c_n')^{p^i}$ ).

To begin with, since 
$b_{i}c_{i+1}^p 
\stackrel{\eqref{equa:aabb}}{=} 
(b_{i+1}c_{i+1})^p 
\stackrel{\eqref{equa:abc}}{=} 
a_{i+1}^p 
\stackrel{\eqref{equa:aabb}}{=} 
a_{i} 
\stackrel{\eqref{equa:abc}}{=} 
b_{i}c_{i}$, we see that $b_i(c_{i+1}^p - c_{i}) = 0$ for all $i$. On the other hand, if $b_{i} \neq 0$, then $b_{i+1}^p \neq 0$ and by Lemma~\ref{lemm:annKer} $(c_{i+2}^p - c_{i+1})^p = 0$. In other words, $b_i \neq 0$ implies 
\begin{equation} \label{equa:cp2}
c_{i+2}^{p^2} = c_{i+1}^p.
\end{equation}
At the beginning we assumed $b_n \neq 0$, so we have \eqref{equa:cp2} for all $i \geq n$.

Set $c_i' := c_{i+2}^{p^2}$ for $i \geq n$. Then for $i \geq n$ we have 
\[ 
a_i 
\stackrel{\eqref{equa:aabb}}{=} 
a_{i+2}^{p^2}
\stackrel{\eqref{equa:abc}}{=} 
b_{i+2}^{p^2}c_{i+2}^{p^2}
\stackrel{def}{=}
b_{i+2}^{p^2}c_i'
\stackrel{\eqref{equa:aabb}}{=} 
b_ic_i'
\]
and 
\[
(c_{i+1}')^p 
\stackrel{def}{=}
(c_{i+3}^{p^2})^p
\stackrel{\eqref{equa:cp2}}{=}
(c_{i+2}^p)^{p}
\stackrel{def}{=}
c_i'
\]
as claimed.
\end{proof}

%\section{The smooth rh topology}
%
%In \cite{Voe10} Voevodsky shows that for a field $k$ admitting resolution of singularities, there is an equivalence of categories $\Shv_{rh}(\Sch_k) \cong \Shv_{rh}(\Sm_k)$ where $Sm_k$ is the category smooth $k$-schemes equipped with the topology induced by the rh topology, \cite[Lem.4.7]{Voe10}. Moreover, he characterises sheaves on $\Sm_k$ as those presheaves which send distinguished Zariski squares and smooth blowup squares to cartesian squares, \cite[Lem.4.6]{Voe10}. This would seem to contradict 

\bibliographystyle{alpha}
\bibliography{bib.bib}

\begin{thebibliography}{{Sta}18}

\bibitem[BS13]{BS13}
Bhargav Bhatt and Peter Scholze.
\newblock The pro-{\'e}tale topology for schemes.
\newblock {\em arXiv preprint arXiv:1309.1198}, 2013.

\bibitem[GK15]{GK15}
Ofer Gabber and Shane Kelly.
\newblock Points in algebraic geometry.
\newblock {\em Journal of Pure and Applied Algebra}, 219(10):4667--4680, 2015.

\bibitem[Gro66]{EGAIV3}
Alexander Grothendieck.
\newblock \'el\'ements de g\'eom\'etrie alg\'ebrique : {IV.} {\'etude} locale
  des sch\'emas et des morphismes de sch\'emas, {Troisi\`eme} partie.
\newblock {\em Publications Math\'ematiques de l'IH\'ES}, 28:5--255, 1966.

\bibitem[MM94]{MM94}
Saunders MacLane and Ieke Moerdijk.
\newblock {\em Sheaves in geometry and logic: A first introduction to topos
  theory}.
\newblock Springer Science \& Business Media, 1994.

\bibitem[MR06]{MR06}
Michael Makkai and Gonzalo~E Reyes.
\newblock {\em First order categorical logic: model-theoretical methods in the
  theory of topoi and related categories}, volume 611.
\newblock Springer, 2006.

\bibitem[SGA72]{SGA41}
{\em Th\'{e}orie des topos et cohomologie \'{e}tale des sch\'{e}mas. {T}ome 1:
  {T}h\'{e}orie des topos}.
\newblock Lecture Notes in Mathematics, Vol. 269. Springer-Verlag, Berlin-New
  York, 1972.
\newblock S\'{e}minaire de G\'{e}om\'{e}trie Alg\'{e}brique du Bois-Marie
  1963--1964 (SGA 4), Dirig\'{e} par M. Artin, A. Grothendieck, et J. L.
  Verdier. Avec la collaboration de N. Bourbaki, P. Deligne et B. Saint-Donat.

\bibitem[{Sta}18]{stacks-project}
The {Stacks Project Authors}.
\newblock \textit{Stacks Project}.
\newblock \url{https://stacks.math.columbia.edu}, 2018.

\end{thebibliography}

%\begin{thebibliography}{9}

%
%\bibitem[Deg06]{Deg06} Fr{\'e}d{\'e}ric D{\'e}glise. \emph{Transferts sur les groupes de {C}how \`a coefficients}. Math. Z. 252, 315–343 (2006).
%
%\bibitem[MS23]{MS}
%Alberto Merici and Shuji Saito. \emph{Cancellation theorems for reciprocity sheaves}. To appear in Algebraic Geometry, 2023. ArXiv preprint arXiv:2001.07902.
%
%\bibitem[LW09]{LcW09}
%Florence Lecomte and Nathalie Wach, \emph{Le complexe motivique de {R}ham},
%  Manuscripta Math. \textbf{129} (2009), no.~1, 75--90. \MR{2496957}
%
%%\bibitem[SV00]{SV00} Andrei Suslin and Vladimir Voevodsky. \emph{Bloch-{K}ato conjecture and motivic cohomology with finite coefficients}. The arithmetic and geometry of algebraic cycles (2000): 117-189.
%
%\bibitem[SV00]{SV00}
%Andrei Suslin, Vladimir Voevodsky. \emph{Relative cycles and {C}how sheaves}. Cycles, transfers, and motivic homology theories, 2000, 143: 10-86.

%\end{thebibliography}
\end{document}